\documentclass[fleqn]{article}

\usepackage{amsmath,amssymb,amsthm,url}
\usepackage{tikz}

\title{On collection schemes and Gaifman's splitting theorem}
\author{Taishi Kurahashi\footnote{Email: kurahashi@people.kobe-u.ac.jp}
\footnote{Graduate School of System Informatics,
Kobe University,
1-1 Rokkodai, Nada, Kobe 657-8501, Japan.}
 and Yoshiaki Minami\footnote{Graduate School of System Informatics,
Kobe University,
1-1 Rokkodai, Nada, Kobe 657-8501, Japan.}}
\date{}

\theoremstyle{plain}
\newtheorem{thm}{Theorem}[section]
\newtheorem{lem}[thm]{Lemma}
\newtheorem{prop}[thm]{Proposition}
\newtheorem{cor}[thm]{Corollary}

\newtheorem{prob}[thm]{Problem}

\theoremstyle{definition}
\newtheorem{defn}[thm]{Definition}

\newtheorem{rem}[thm]{Remark}

\newcommand{\PA}{\mathbf{PA}}
\newcommand{\Coll}{\mathbf{Coll}}
\newcommand{\CS}{\Coll_{\mathrm{s}}}
\newcommand{\CSm}{\CS^-}
\newcommand{\Cm}{\Coll^-}
\newcommand{\Cd}{\Coll^{\mathrm{d}}}
\newcommand{\CW}{\Coll_{\mathrm{w}}}
\newcommand{\CWm}{\CW^-}
\newcommand{\CWd}{\CW^{\mathrm{d}}}

\newcommand{\LA}{\mathcal{L}_A}

\newcommand{\ElemDiag}{\mathrm{ElemDiag}}

\newcommand{\ee}{\subseteq_{\mathrm{end}}}
\newcommand{\cf}{\subseteq_{\mathrm{cof}}}
\newcommand{\dcf}{\prec_{\Delta_0}^{\mathrm{cof}}}

\newcommand{\cof}[1]{\mathsf{cof}_{#1}}
\newcommand{\COF}[1]{\mathsf{COF}_{#1}}
\newcommand{\eex}[1]{\mathsf{end}_{#1}}
\newcommand{\cofm}[1]{\mathsf{cof}_{#1}^{\equiv}}
\newcommand{\COFm}[1]{\mathsf{COF}_{#1}^{\equiv}}
\newcommand{\eexm}[1]{\mathsf{end}_{#1}^{\equiv}}
\newcommand{\coff}[1]{\mathsf{cof}_{#1}^{<}}
\newcommand{\COFf}[1]{\mathsf{COF}_{#1}^{<}}
\newcommand{\eexf}[1]{\mathsf{end}_{#1}^{<}}

\newcommand{\mc}[1]{\mathcal{#1}}

\begin{document}

\maketitle

\begin{abstract}
We study model theoretic characterizations of various collection schemes over $\mathbf{PA}^-$ from the viewpoint of Gaifman's splitting theorem. 
Among other things, we prove that for any $n \geq 0$ and $M \models \PA^-$, the following are equivalent: 
\begin{enumerate}
	\item $M$ satisfies the collection scheme for $\Sigma_{n+1}$ formulas.  
	\item For any $K, N \models \PA^-$, if $M \cf K$, $M \prec_{\Delta_0} K$ and $M \prec N$, then $M \prec_{\Sigma_{n+2}} K$ and $\sup_N(M) \prec_{\Sigma_n} N$.  
	\item For any $N \models \PA^-$, if $M \prec N$, then $M \prec_{\Sigma_{n+2}} \sup_N(M) \prec_{\Sigma_{n}} N$. 
\end{enumerate}
Here, $\sup_N(M)$ is the unique $K$ satisfying $M \cf K \ee N$. 
We also investigate strong collection schemes and parameter-free collection schemes from the similar perspective. 
\end{abstract}

\section{Introduction}

The language $\LA$ of first-order arithmetic consists of constant symbols $0$ and $1$, binary function symbols $+$ and $\times$, and binary relation symbol $<$. 
The $\LA$-theory of the non-negative parts of commutative discretely ordered rings is denoted by $\PA^-$ (Kaye \cite[Chapter 2]{Kaye}).

Let $\vec{v}$ denote a finite sequence of variables allowing the empty sequence. 
The following definition introduces some variations of the collection scheme, which have appeared in the literature so far.

\begin{defn}\label{colls}Let $\Gamma$ be a class of $\LA$-formulas. 
\begin{itemize}
	\item $\Coll(\Gamma)$ is the scheme
	\[
		\forall \vec{z} \, \forall \vec{u}\, \bigl(\forall \vec{x} < \vec{u}\, \exists \vec{y}\, \varphi(\vec{x}, \vec{y}, \vec{z}) \to \exists \vec{v}\, \forall \vec{x} < \vec{u}\, \exists \vec{y} < \vec{v}\, \varphi(\vec{x}, \vec{y}, \vec{z}) \bigr), \quad \varphi \in \Gamma. 
	\]
	\item $\Cd(\Gamma)$ is the scheme
	\[
		\forall \vec{u}\, \bigl(\forall \vec{z}\, \forall \vec{x} < \vec{u}\, \exists \vec{y}\, \varphi(\vec{x}, \vec{y}, \vec{z}) \to \forall \vec{z} \, \exists \vec{v}\, \forall \vec{x} < \vec{u}\, \exists \vec{y} < \vec{v}\, \varphi(\vec{x}, \vec{y}, \vec{z}) \bigr), \quad \varphi \in \Gamma. 
	\]
	\item $\Cm(\Gamma)$ is the scheme
	\[
		\forall \vec{u}\, \bigl(\forall \vec{x} < \vec{u}\, \exists \vec{y}\, \varphi(\vec{x}, \vec{y}) \to \exists \vec{v}\, \forall \vec{x} < \vec{u}\, \exists \vec{y} < \vec{v}\, \varphi(\vec{x}, \vec{y}) \bigr), \quad \varphi \in \Gamma. 
	\]
	\item $\CW(\Gamma)$ is the scheme
	\[
		\forall \vec{z} \, \bigl(\forall \vec{x} \, \exists \vec{y}\, \varphi(\vec{x}, \vec{y}, \vec{z}) \to \forall \vec{u}\, \exists \vec{v}\, \forall \vec{x} < \vec{u}\, \exists \vec{y} < \vec{v}\, \varphi(\vec{x}, \vec{y}, \vec{z}) \bigr), \quad \varphi \in \Gamma. 
	\]
	\item $\CWd(\Gamma)$ is the scheme
	\[
		\forall \vec{z}\, \forall \vec{x} \, \exists \vec{y}\, \varphi(\vec{x}, \vec{y}, \vec{z}) \to \forall \vec{z} \, \forall \vec{u}\, \exists \vec{v}\, \forall \vec{x} < \vec{u}\, \exists \vec{y} < \vec{v}\, \varphi(\vec{x}, \vec{y}, \vec{z}), \quad \varphi \in \Gamma. 
	\]
	\item $\CWm(\Gamma)$ is the scheme
	\[
		\forall \vec{x} \, \exists \vec{y}\, \varphi(\vec{x}, \vec{y}) \to \forall \vec{u}\, \exists \vec{v}\, \forall \vec{x} < \vec{u}\, \exists \vec{y} < \vec{v}\, \varphi(\vec{x}, \vec{y}), \quad \varphi \in \Gamma. 
	\]
	\item $\CS(\Gamma)$ is the scheme
	\[
		\forall \vec{z} \, \forall \vec{u}\, \exists \vec{v}\, \forall \vec{x} < \vec{u}\, \bigl(\exists \vec{y}\, \varphi(\vec{x}, \vec{y}, \vec{z}) \to \exists \vec{y} < \vec{v}\, \varphi(\vec{x}, \vec{y}, \vec{z}) \bigr), \quad \varphi \in \Gamma. 
	\]
	\item $\CSm(\Gamma)$ is the scheme
	\[
		\forall \vec{u}\, \exists \vec{v}\, \forall \vec{x} < \vec{u}\, \bigl(\exists \vec{y}\, \varphi(\vec{x}, \vec{y}) \to \exists \vec{y} < \vec{v}\, \varphi(\vec{x}, \vec{y}) \bigr), \quad \varphi \in \Gamma. 
	\]
\end{itemize}
\end{defn}

In the literature, the collection schemes have been usually considered together with some induction scheme. 
For a class $\Gamma$ of $\LA$-formulas, let $\mathbf{I}\Gamma$ denote the $\LA$-theory obtained from $\PA^-$ by adding the scheme of induction for formulas in $\Gamma$. 
Peano arithmetic $\PA$ is defined as the union of $\{\mathbf{I}\Sigma_n \mid n \geq 0\}$. 

The purpose of the present paper is to study model theoretic characterizations of these collection schemes over $\PA^-$ from the viewpoint of Gaifman's splitting theorem. 
For each $M, N \models \PA^-$ with $M \subseteq N$, let $\sup_N(M)$ be the unique $K \models \PA^-$ such that $M \cf K \ee N$. 
Here, $M \cf K$ and $K \ee N$ stand for `$K$ is a cofinal extension of $M$' and `$N$ is an end-extension of $K$', respectively. 
Also, let $M \prec K$ stand for that $K$ is an elementary extension of $M$. 
Gaifman's splitting theorem \cite[Theorem 4]{Gaif} is a basic result concerning the structure $\sup_N(M)$, which states that if $M, N \models \PA$ and $M \subseteq N$, then $M \prec \sup_N(M)$. 
Relating to this theorem, Kaye \cite[Theorem 3.2]{Kaye91} proved that if $M, N \models \mathbf{I}\Sigma_n$ and $M \prec_{\Sigma_n} N$, then $\sup_N(M) \prec_{\Sigma_n} N$. 
Here, $M \prec_\Gamma N$ stands for that $N$ is a $\Gamma$-elementary extension of $M$. 
It immediately follows from these results that if $M$ is a model of $\PA$, then for any $N \models \PA^-$ with $M \prec N$, we have $M \prec \sup_N(M) \prec N$. 
Kaye also proved that the converse implication of the last statement also holds, that is, models of $\PA$ can be characterized by considering $\sup_N(M)$ as follows:  

\begin{thm}[Kaye {\cite[Theorem 1.4]{Kaye91}}]\label{Kaye2}
Suppose $M \models \mathbf{I}\Delta_{0}$.
Then, $M \models \PA$ if and only if for any $N \models \PA^-$, if $M \prec N$, then $M \prec \sup_N(M) \prec N$. 
\end{thm}

Our main aim is to improve Theorem \ref{Kaye2} from two points of view.
First, we stratify Theorem \ref{Kaye2} based on arithmetical hierarchy. 
This stratification shows that the various collection schemes are characterized by properties concerning $\sup_N(M)$.
Second, we weaken $\mathbf{I}\Delta_0$ in the statement of Theorem \ref{Kaye2} to $\PA^-$. 
This weakening shows that induction axioms are not directly involved in our characterization results.

Among other things, we actually prove the following equivalences: For any $n \geq 0$ and $M \models \PA^-$, 
\begin{itemize}
	\item $M \models \CS(\Sigma_{n+1})$ if and only if for any $N \models \PA^-$, if $M \prec N$, then $\sup_N(M) \prec_{\Sigma_{n+1}} N$.  \hfill (Theorem \ref{MT1})
	\item $M \models \Coll(\Sigma_{n+1})$ if and only if for any $N \models \PA^-$, if $M \prec N$, then $M \prec_{\Sigma_{n+2}} \sup_N(M) \prec_{\Sigma_{n}} N$. \hfill (Theorem \ref{MT2})
\end{itemize}

In addition to these characterization results, we also prove the following results on $\Delta_0$-elementary cofinal extensions: For any $n \geq 0$ and $M, K \models \PA^-$ with $M \prec_{\Delta_0} K$ and $M \cf K$, 
\begin{itemize}
	\item If $M \models \Coll(\Sigma_{n+1})$, then $M \prec_{\Sigma_{n+2}} K$. \hfill (Theorem \ref{Appl1}.1)
	\item If $M \models \CS(\Sigma_{n+1})$, then $K \models \CS(\Sigma_{n+1})$.  \hfill (Theorem \ref{Appl1}.2)
\end{itemize}

Our proofs of these results are mostly based on the compactness argument, which is not so deep, but we believe that our results provide some insight into the relationship between collection schemes, Gaifman's splitting theorem, and cofinal extensions.

The organization of the present paper is as follows. 
Section \ref{Sec:pre} is devoted to preliminaries. 
We provide a detailed background of our work and introduce some notions and their basic properties.
In Section \ref{Sec:SC}, we prove our characterization theorem for $\CS(\Sigma_{n+1})$. 
We also prove the above mentioned results concerning $\Delta_0$-elementary cofinal extensions. 
In Section \ref{Sec:Coll}, we prove our characterization theorem for $\Coll(\Sigma_{n+1})$. 
Sections \ref{Sec:WColl-} and \ref{Sec:Coll-} are devoted to similar investigations of the parameter-free collection schemes $\CWm(\Sigma_{n+1})$ and $\Cm(\Sigma_{n+1})$, respectively. 
Finally, in Section \ref{Sec:versus}, we discuss the equivalence between the property $\cof{n}$ concerning $\sup_N(M)$ and the property $\COF{n}$ concerning more general cofinal extensions.

\section{Preliminaries}\label{Sec:pre}

This section consists of two subsections. 
The first subsection provides the sources of the various collection schemes introduced in Definition \ref{colls}. 
In the second subsection, we provide a detailed background of our study, introduce some notions and give their basic properties.

\subsection{Variations of the collection scheme}

The classes $\Delta_0$, $\Sigma_n$, and $\Pi_n$ of $\LA$-formulas are introduced in the usual way (cf.~\cite[Chapter 7]{Kaye}). 
It is clear that each scheme of Definition \ref{colls} with $\Gamma = \Pi_n$ for $n \geq 0$ is deductively equivalent to the scheme of the same type with $\Gamma = \Sigma_{n+1}$. 
For instance, $\CS(\Pi_n)$ is equivalent to $\CS(\Sigma_{n+1})$. 
So, this paper deals with only the collection schemes of Definition \ref{colls} with $\Gamma = \Sigma_{n}$. 

We begin with a brief review of the sources of the various collection schemes.
\begin{itemize}
	\item Parsons \cite{Pars} studied the scheme $\Coll(\Sigma_n)$ over theories of arithmetic having some induction scheme and proved that the theory $\mathbf{I}\Sigma_n$ proves $\Coll(\Sigma_n)$ (cf.~\cite[Lemmas 2 and 3]{Pars}). 
Paris and Kirby \cite{PK} introduced the theory $\mathbf{B}\Sigma_n : = \mathbf{I}\Delta_0 + \Coll(\Sigma_n)$ and investigated the properties of the theory from a model theoretic point of view. 

 \item For the scheme $\CW(\Sigma_n)$, the subscript `w' stands for `weak', but it is easy to see that $\Coll(\Sigma_n)$ and $\CW(\Sigma_n)$ are equivalent over $\PA^-$ (Proposition \ref{OProp2}). 
 For example, $\mathbf{I}\Delta_0 + \CW(\Sigma_n)$ is denoted by $\mathbf{B}\Sigma_n$ in Kaye, Paris and Dimitracopoulos' paper \cite[p.~1082]{KPD}. 

 \item The main purpose of the paper \cite{KPD} was to analyze the strength of the parameter-free versions of the induction and collection schemes. 
	In that paper, the theory $\mathbf{B}\Sigma_n^- : = \mathbf{I}\Delta_0 + \CWm(\Sigma_n)$ was introduced and it is shown that $\mathbf{B}\Sigma_{n+1}^- \vdash \mathbf{I}\Sigma_n$ (cf.~\cite[Proposition 1.2]{KPD}). 
It is not known if the theories $\mathbf{I}\Delta_0 + \Cm(\Sigma_n)$ and $\mathbf{B}\Sigma_n^-$ are deductively equivalent (cf.~\cite[p.~1097]{KPD} and \cite[Problem 2.1]{CFL}). 
    The theory $\mathbf{I}\Delta_0 + \Cm(\Sigma_n)$ is denoted by $\mathbf{B}_s(\Sigma_n)$ in Cord\'on-Franco et al.~\cite{CFL}, but we do not adopt this notation to avoid confusion with the notation for strong collection schemes. 

	\item Of course the parameter-free version of a scheme is weaker than the original one and the scheme having the superscript $\mathrm{d}$ is intermediate between them. 
That is, $\Coll(\Sigma_n) \vdash \Cd(\Sigma_n) \vdash \Cm(\Sigma_n)$ and $\CW(\Sigma_n) \vdash \CWd(\Sigma_n) \vdash \CWm(\Sigma_n)$ hold. 
	The superscript $\mathrm{d}$ here stands for `distributed' because $\Cd(\Sigma_n)$ and $\CWd(\Sigma_n)$ are respectively obtained from $\Coll(\Sigma_n)$ and $\CW(\Sigma_n)$ by distributing the quantifiers $\forall \vec{z}$ in the schemes. 
	The scheme $\CWd(\Sigma_n)$ was considered in \cite[Exercise 10.3]{Kaye}, where the theory $\mathbf{I}\Delta_0 + \CWd(\Sigma_n)$ is denoted by $\mathbf{B}\Sigma_n^-$. 
 
	\item The scheme $\CS$ is known as the strong collection scheme because $\CS(\Gamma)$ is stronger than $\Coll(\Gamma)$ (Proposition \ref{OProp1}). 
	The theory $\mathbf{S}\Sigma_n : = \mathbf{I}\Delta_0 + \CS(\Sigma_n)$ was considered in H\'ajek and Pudl\'ak \cite{HP}, and interestingly, it is known that $\mathbf{S}\Sigma_{n+1}$ is deductively equivalent to $\mathbf{I}\Sigma_{n+1}$ (cf.~\cite[Theorem 2.23]{HP} and \cite[Lemma 10.6 and Exercise 10.6]{Kaye}). 
    It is easy to see that $\CS(\Sigma_n)$ is equivalent to its parameter-free version (Proposition \ref{OProp3}). 
\end{itemize}

We list some easily verifiable facts on collection schemes. 

\begin{prop}\label{OProp1}
For any $n \geq 0$, $\PA^- + \CS(\Sigma_{n}) \vdash \Coll(\Sigma_{n})$. 
\end{prop}

\begin{prop}\label{OProp2}
For any $n \geq 0$, $\PA^- + \Coll(\Sigma_{n})$ is deductively equivalent to $\PA^- + \CW(\Sigma_{n})$. 
\end{prop}

\begin{prop}\label{OProp3}
For any $n \geq 0$, $\PA^- + \CS(\Sigma_{n})$ is deductively equivalent to $\PA^- + \CSm(\Sigma_{n})$. 
\end{prop}

\begin{prop}[Cf.~Kaye {\cite[Proposition 7.1]{Kaye}}]
Let $n \geq 0$. 
\begin{enumerate}
    \item For any $\Sigma_{n+1}$ formula $\varphi(\vec{x}, \vec{y})$, the formula $\forall \vec{y} < \vec{z}\, \varphi(\vec{x}, \vec{y})$ is provably equivalent to some $\Sigma_{n+1}$ formula over $\PA^- + \Coll(\Sigma_{n+1})$. 
    \item For any $\Pi_{n+1}$ formula $\varphi(\vec{x}, \vec{y})$, the formula $\exists \vec{y} < \vec{z}\, \varphi(\vec{x}, \vec{y})$ is provably equivalent to some $\Pi_{n+1}$ formula over $\PA^- + \Coll(\Sigma_{n+1})$. 
\end{enumerate}
\end{prop}

It is proved in \cite[Proposition 1.2]{KPD} that $\mathbf{B}\Sigma_{n+1}^- \vdash \mathbf{I}\Sigma_n$ for each $n \geq 0$ and is improved as follows.  

\begin{prop}\label{w_to_s}
For each $n \geq 1$, $\PA^- + \CWm(\Sigma_{n+1}) \vdash \CS(\Sigma_{n})$. 
\end{prop}
\begin{proof}
By Proposition \ref{OProp3}, it suffices to prove $\PA^- + \CWm(\Sigma_{n+1}) \vdash \CSm(\Sigma_{n})$. 
Let $\varphi(\vec{x}, \vec{y})$ be any $\Sigma_{n}$ formula. 
By logic, we have
\[
    \vdash \forall \vec{x} \, \exists \vec{y} \bigl(\exists \vec{y}\, \varphi(\vec{x}, \vec{y}) \to \varphi(\vec{x}, \vec{y}) \bigr).  
\]
Since $\exists \vec{y}\, \varphi(\vec{x}, \vec{y}) \to \varphi(\vec{x}, \vec{y})$ is logically equivalent to some $\Sigma_{n+1}$ formula, $\PA^- + \CWm(\Sigma_{n+1})$ proves
\begin{align*}
     & \forall \vec{x} \, \exists \vec{y} \, \bigl(\exists \vec{y}\, \varphi(\vec{x}, \vec{y}) \to \varphi(\vec{x}, \vec{y}) \bigr) \\
    & \quad \quad \to \forall \vec{u}\, \exists \vec{v}\, \forall \vec{x} < \vec{u}\, \exists \vec{y} < \vec{v}\, \bigl(\exists \vec{y}\, \varphi(\vec{x}, \vec{y}) \to \varphi(\vec{x}, \vec{y}) \bigr).
\end{align*}
Thus,
\[
    \PA^- + \CWm(\Sigma_{n+1}) \vdash \forall \vec{u}\, \exists \vec{v}\, \forall \vec{x} < \vec{u}\, \exists \vec{y} < \vec{v}\, \bigl(\exists \vec{y}\, \varphi(\vec{x}, \vec{y}) \to \varphi(\vec{x}, \vec{y}) \bigr). 
\]
Equivalently, 
\[
    \PA^- + \CWm(\Sigma_{n+1}) \vdash \forall \vec{u}\, \exists \vec{v}\, \forall \vec{x} < \vec{u}\, \bigl(\exists \vec{y}\, \varphi(\vec{x}, \vec{y}) \to \exists \vec{y} < \vec{v}\, \varphi(\vec{x}, \vec{y}) \bigr). \qedhere
\]
\end{proof}

It is known that the theory $\PA^- + \bigcup_{n \in \omega} \Coll(\Sigma_n)$ having the full collection scheme does not prove $\mathbf{I}\Delta_0$ (cf.~\cite[Exercise 7.7]{Kaye}). 
Furthermore, it can be shown that $\PA^- + \bigcup_{n \in \omega} \Coll(\Sigma_n)$ is $\Pi_1$-conservative over $\PA^-$, and so even $\PA^- + \bigcup_{n \in \omega} \Coll(\Sigma_n) \nvdash \mathbf{IOpen}$ holds. 
In the study of the collection schemes, what properties of the collection schemes can be shown without using the induction axioms?
The right hand side of the dashed line of Figure \ref{Fig1} suggests the possibility of analyzing the situations of the collection schemes over the theory $\PA^-$ without induction axioms. 
In the present paper, we follow this suggestion and show relationships between several variants of collection schemes over the theory $\PA^-$. 

\begin{figure}[ht]
\centering
\begin{tikzpicture}
\node (ind) at (0,3) {$\mathbf{I}\Sigma_{n+1}$};
\node (col) at (0, 2) {$\mathbf{B}\Sigma_{n+1}$};
\node (colm) at (0,1) {$\mathbf{B}\Sigma_{n+1}^-$};
\node (ind-) at (0,0) {$\mathbf{I}\Sigma_{n}$};

\node at (1, 3) {$=$};
\node at (1, 2) {$=$};
\node at (1, 1) {$=$};
\node at (1, 0) {$=$};

\node (DS) at (3, 3) {$\mathbf{I}\Delta_0 + \CS(\Sigma_{n+1})$};
\node (DN) at (3, 2) {$\mathbf{I}\Delta_0 + \Coll(\Sigma_{n+1})$};
\node (DW) at (3, 1) {$\mathbf{I}\Delta_0 + \CWm(\Sigma_{n+1})$};
\node (Ds) at (3, 0) {$\mathbf{I}\Delta_0 + \CS(\Sigma_n)$};

\node (S) at (7, 3) {$\PA^- + \CS(\Sigma_{n+1})$};
\node (N) at (7, 2) {$\PA^- + \Coll(\Sigma_{n+1})$};
\node (W) at (7, 1) {$\PA^- + \CWm(\Sigma_{n+1})$};
\node (s) at (7, 0) {$\PA^- + \CS(\Sigma_n)$};

\draw [->, double] (ind)--(col);
\draw [->, double] (col)--(colm);
\draw [->, double] (colm)--(ind-);
\draw [->, double] (DS)--(S);
\draw [->, double] (DN)--(N);
\draw [->, double] (DW)--(W);
\draw [->, double] (Ds)--(s);
\draw [->, double] (S)--(N);
\draw [->, double] (N)--(W);
\draw [->, double] (W)--(s);

\draw [dashed] (5, -0.5)--(5, 3.5);

\end{tikzpicture}
\caption{The relationships between induction and collection schemes}\label{Fig1}
\end{figure}

\subsection{Model theoretic viewpoint}

\begin{defn}
Let $M, K \models \PA^-$ be such that $M \subseteq K$ and $\Gamma$ be a class of formulas. 
\begin{itemize}
    \item We say that $K$ is an \textit{end-extension} of $M$ (denoted by $M \ee K$) iff for any $a, b \in K$, if $b \in M$ and $K \models a < b$, then $a \in M$. 
    
    \item We say that $K$ is a \textit{cofinal extension} of $M$ (denoted by $M \cf K$) iff for any $a \in K$, there exists a $b \in M$ such that $K \models a < b$. 
    
    \item We say that $K$ is a \textit{$\Gamma$-elementary extension} of $M$ (denoted by $M \prec_{\Gamma} K)$ iff for any $\vec{a} \in M$ and any $\Gamma$ formula $\varphi(\vec{x})$, we have $M \models \varphi(\vec{a})$ if and only if $K \models \varphi(\vec{a})$. 

    \item We say that $K$ is a \textit{$\Gamma$-elementary cofinal extension} of $M$ (denoted by $M \prec_{\Gamma}^{\mathrm{cof}} K)$ iff $M \prec_{\Gamma} K$ and $M \cf K$. 
\end{itemize}
\end{defn}

Paris and Kirby \cite{PK} established the following model theoretic characterization of the collection scheme: 

\begin{thm}[Paris and Kirby {\cite[Theorem B]{PK}}]\label{PK}
Let $M$ be any model of $\PA^-$. 
\begin{enumerate}
	\item For $n \geq 1$, if $M$ has a proper $\Sigma_n$-elementary end-extension, then $M \models \Coll(\Sigma_n)$. 
	\item For $n \geq 2$, if $M$ is a countable model of $\mathbf{B}\Sigma_n$, then $M$ has a proper $\Sigma_n$-elementary end-extension. 
\end{enumerate}
\end{thm}

Also, the following sufficient condition for a model of $\PA^-$ to satisfy $\mathbf{B}\Sigma_{n}$ is known. 

\begin{thm}\label{Clote}
Let $M$ be any model of $\PA^-$. 
\begin{enumerate}
	\item \textup{(Wilkie and Paris {\cite[Theorem 1]{WP}})} If $M$ has a proper end-extension $N \models \mathbf{I}\Delta_0$, then $M \models \mathbf{B}\Sigma_1$. 
	\item \textup{(Clote {\cite[Proposition 3]{Clot}}; Paris and Kirby {\cite[Theorem B]{PK}} for $n = 1$)} For $n \geq 1$, if $M$ has a proper $\Sigma_n$-elementary end-extension $N \models \mathbf{I}\Sigma_{n-1}$, then $M \models \mathbf{B}\Sigma_{n+1}$. 
\end{enumerate}
\end{thm}

The theory $\mathbf{I}\Delta_0$ plays an essential role in these results. 
For example, Theorem \ref{Clote}.(1) is no longer true if we weaken the condition `$N \models \mathbf{I}\Delta_0$' to `$N \models \PA^-$' because every $M \models \PA^-$ has a proper end-extension $N \models \PA^-$ (cf.~\cite[Exercise 7.7]{Kaye}), and there exists a model of $\PA^-$ in which $\mathbf{B}\Sigma_1$ does not hold. 
Since we also want to analyze the properties of collection schemes in models that do not necessarily satisfy $\mathbf{I}\Delta_0$, we should consider phenomena in a different fashion from these results. 
We will therefore focus on cofinal extensions. 
Gaifman's splitting theorem is a basic result for the cofinal extensions of models of $\PA$. 

\begin{defn}
For $M, N \models \PA^-$ with $M \subseteq N$, let $\sup_N(M) : = \{a \in N \mid (\exists b \in M) \ N \models a \leq b\}$. 
\end{defn}

It is clear that $\sup_N(M)$ is the unique $K \models \PA^-$ such that $M \cf K \ee N$ for each $M, N \models \PA^-$ with $M \subseteq N$. 

\begin{thm}[Gaifman's splitting theorem {\cite[Theorem 4]{Gaif}}]\label{Splitting}
If $M, N \models \PA$ and $M \subseteq N$, then $M \prec \sup_N(M)$. 
\end{thm}

Gaifman's splitting theorem follows from the following theorem: 

\begin{thm}[Gaifman {\cite[Theorem 3]{Gaif}}]\label{Gaifman1}
Let $M, K \models \PA^-$. 
If $M \models \PA$ and $M \dcf K$, then $M \prec K$. 
\end{thm}

The proof of Theorem \ref{Gaifman1} presented in the textbook of Kaye \cite{Kaye} actually proves the following hierarchical refinement.

\begin{thm}[Cf.~Kaye {\cite[Theorem 7.7]{Kaye}}]\label{Kaye1}
Let $M, K \models \PA^-$. 
\begin{enumerate}
    \item If $M \models \Coll(\Sigma_1)$ and $M \dcf K$, then $M \prec_{\Sigma_{2}} K$. 

    \item For $n \geq 1$, if $M \models \mathbf{B}\Sigma_{n+1}$ and $M \dcf K$, then $M \prec_{\Sigma_{n+2}} K$. 
\end{enumerate}
\end{thm}

In the proof of the second clause of Theorem \ref{Kaye1} presented in \cite{Kaye}, the principle of the finite axiom of choice $\mathbf{FAC}(\Sigma_{n+1})$ for $\Sigma_{n+1}$ formulas is actually used instead of $\mathbf{B}\Sigma_{n+1}$, but it is known that $\mathbf{FAC}(\Sigma_{n+1})$ is equivalent to $\mathbf{B}\Sigma_{n+1}$ for $n \geq 1$ over $\mathbf{I}\Delta_0$ (cf.~H\'ajek and Pudl\'ak \cite{HP}).

These phenomena regarding cofinal extensions are clearly related to collection axioms, and indeed, these results are presented in the chapter on collection in Kaye's book \cite[Chapter 7]{Kaye}. 
Relating to Gaifman's splitting theorem, Mijajlovi{\'c} proved the following result concerning the relation between $\sup_N(M)$ and $N$. 

\begin{thm}[Mijajlovi{\'c} {\cite[Theorem 1.2]{Mija}}]\label{Mijajlovic1}
Let $M, N \models \PA$ and $n \geq 0$. 
If $M \prec_{\Sigma_n} N$, then $\sup_N(M) \prec_{\Sigma_{n}} N$. 
\end{thm}

The proof of Theorem 3.2 of Kaye \cite{Kaye91} refines Mijajlovi{\'c}'s theorem as follows: 

\begin{thm}[Kaye {\cite[Theorem 3.2]{Kaye91}}]\label{Kaye3}
Let $M, N \models \mathbf{I}\Sigma_n$ and $n \geq 0$. 
If $M \prec_{\Sigma_n} N$, then $\sup_N(M) \prec_{\Sigma_{n}} N$.\footnote{For $n = 0$, the assumption $M, N \models \mathbf{I}\Delta_0$ can of course be weakened to $M, N \models \PA^-$.}
\end{thm}

It follows from Theorems \ref{Gaifman1} and \ref{Mijajlovic1} that for any $M, N \models \PA^-$, if $M \models \PA$ and $M \prec N$, then $M \prec \sup_N(M) \prec N$. 
Kaye proved that the converse of this statement also holds in the following sense. 

\setcounter{section}{1}
\setcounter{thm}{1}

\begin{thm}[Kaye {\cite[Theorem 1.4]{Kaye91}}(restated)]
For $M \models \mathbf{I}\Delta_{0}$, the following are equivalent: 
\begin{enumerate}
    \item $M \models \PA$.
    \item For any $N \models \PA^-$, if $M \prec N$, then $M \prec \sup_N(M) \prec N$. 
\end{enumerate}
\end{thm}

\setcounter{section}{2}
\setcounter{thm}{14}

Inspired by Theorems \ref{Kaye1}, \ref{Mijajlovic1}, \ref{Kaye3} and \ref{Kaye2}, we introduce the following properties on models, which are our main research interests. 

\begin{defn}
Let $M \models \PA^-$ and $n \geq 0$. 
\begin{itemize}
    \item We say that $M$ satisfies the condition $\eex{n}$ iff for any $N \models \PA^-$, if $M \prec N$, then $\sup_N(M) \prec_{\Sigma_{n}} N$. 

    \item We say that $M$ satisfies the condition $\cof{n}$ iff for any $N \models \PA^-$, if $M \prec N$, then $M \prec_{\Sigma_{n}} \sup_N(M)$. 

    \item We say that $M$ satisfies the condition $\COF{n}$ iff for any $K \models \PA^-$, if $M \dcf K$, then $M \prec_{\Sigma_{n}} K$. 

\end{itemize}
\end{defn}

It is easy to see that every model satisfying $\COF{n}$ also satisfies $\cof{n}$. 
Notice that every model of $\PA^-$ trivially satisfies $\eex{0}$. 
Also, every model satisfies $\COF{1}$. 

\begin{prop}\label{GA0}
Every model of $\PA^-$ satisfies $\COF{1}$. 
\end{prop}
\begin{proof}
Let $M, K \models \PA^-$ be such that $M \dcf K$. 
Let $\vec{a}$ be any elements of $M$ and $\varphi(\vec{x}, \vec{y})$ be any $\Delta_0$ formula such that $K \models \exists \vec{x}\, \varphi(\vec{x}, \vec{a})$. 
Since $M \cf K$, we find some $\vec{b} \in M$ such that $K \models \exists \vec{x} < \vec{b}\, \varphi(\vec{x}, \vec{a})$. 
Since $M \prec_{\Delta_0} K$, we have $M \models \exists \vec{x} < \vec{b}\, \varphi(\vec{x}, \vec{a})$, and hence $M \models \exists \vec{x}\, \varphi(\vec{x}, \vec{a})$. 
Thus $M \prec_{\Sigma_1} K$ holds. 
\end{proof}

In general, $\eex{n}$ implies $\cof{n+1}$. 

\begin{prop}\label{eexcof}
For any $n \geq 0$ and $M \models \PA^-$, if $M$ satisfies $\eex{n}$, then $M$ also satisfies $\cof{n+1}$. 
\end{prop}
\begin{proof}
Suppose that $M$ satisfies $\eex{n}$. 
Let $N \models \PA^-$ be any model such that $M \prec N$. 
Let $\varphi(\vec{x})$ be any $\Sigma_{n+1}$ formula and $\vec{a}$ be any tuple of elements of $M$. 
Suppose $\sup_N(M) \models \varphi(\vec{a})$. 
By the condition $\eex{n}$, we have $\sup_N(M) \prec_{\Sigma_n} N$, and so $N \models \varphi(\vec{a})$. 
Since $M \prec N$, we get $M \models \varphi(\vec{a})$. 
Thus, we have proved $M \prec_{\Sigma_{n+1}} \sup_N(M)$. 
\end{proof}

Theorem \ref{Kaye1} says that every model of $\Coll(\Sigma_1)$ satisfies $\COF{2}$, and that for $n \geq 1$, every model of $\mathbf{B}\Sigma_{n+1}$ satisfies $\COF{n+2}$. 

In the present paper, we show that the properties $\eex{n}$, $\cof{n}$ and $\COF{n}$ exactly capture the behavior of several collection schemes over models of $\PA^-$. 
Our main results are as follows: For any $n \geq 0$ and $M \models \PA^-$, 
\begin{itemize}
	\item $M \models \CS(\Sigma_{n+1})$ if and only if $M$ satisfies $\eex{n+1}$. \hfill (Theorem \ref{MT1})
	\item $M \models \Coll(\Sigma_{n+1})$ if and only if $M$ satisfies $\eex{n}$ and $\cof{n+2}$ if and only if $M$ satisfies $\eex{n}$ and $\COF{n+2}$. \hfill (Theorem \ref{MT2})
\end{itemize}

Furthermore, we will introduce the two conditions $\cofm{n}$ and $\coff{n}$ that are variants of $\cof{n}$, and prove similar results for other collection schemes by using these conditions. 
The implications and equivalences obtained in the present paper are summarized in Figure \ref{Fig2}.  

\begin{figure}[ht]
\centering
\begin{tikzpicture}
\node (COLW-) at (0,6.5) {$\CWd(\Sigma_{n+2})$};
\node (COLW) at (4, 6.5) {$\CWm(\Sigma_{n+2})$};
\node (strong-) at (0,5.2) {$\CSm(\Sigma_{n+1})$};
\node (strong) at (4, 5.2) {$\CS(\Sigma_{n+1})$};
\node (colw) at (0,3.9) {$\CW(\Sigma_{n+1})$};
\node (col) at (4,3.9) {$\Coll(\Sigma_{n+1})$};
\node (cold) at (0,2.6) {$\Cd(\Sigma_{n+1})$};
\node (col-) at (4,2.6) {$\Cm(\Sigma_{n+1})$};
\node (colwd) at (0,1.3) {$\CWd(\Sigma_{n+1})$};
\node (colw-) at (4,1.3) {$\CWm(\Sigma_{n+1})$};

\node (MIGAm') at (9, 6.5) {$\eex{n+1}$\ \&\ $\cofm{n+3}$};
\node (MIn+1) at (9, 5.2) {$\eex{n+1}$};
\node (MIGA) at (9, 3.9) {$\eex{n}$\ \&\ $\cof{n+2}$};
\node (MIGAf) at (9, 2.6) {$\eex{n}$\ \&\ $\coff{n+2}$};
\node (MIGAm) at (9, 1.3) {$\eex{n}$\ \&\ $\cofm{n+2}$};

\draw [<->, double] (COLW-)--(COLW);
\draw [->, double] (COLW)--(strong);
\draw [<->, double] (strong)--(strong-);
\draw [->, double] (strong)--(col);
\draw [<->, double] (col)--(colw);
\draw [->, double] (col)--(col-);
\draw [<->, double] (cold)--(col-);
\draw [->, double] (col-)--(colw-);
\draw [<->, double] (colwd)--(colw-);

\draw [->, double] (4, 7.3)--(COLW);
\draw [->, double] (colw-)--(4, 0.5);

\draw [<->, double] (COLW)--(MIGAm');
\draw [<->, double] (strong)--(MIn+1);
\draw [<->, double] (col)--(MIGA);
\draw [<->, double] (col-)--(MIGAf);
\draw [<->, double] (colw-)--(MIGAm);

\draw[dashed] (5.5, 0.5)--(5.5,7.3);

\draw (4,5.9) node[left]{Prop.~\ref{w_to_s}};

\draw (4,4.6) node[left]{Prop.~\ref{OProp1}};

\draw (2,5.2) node[above]{Prop.~\ref{OProp3}};
\draw (6.5,5.2) node[above]{Thm.~\ref{MT1}};

\draw (2,3.9) node[above]{Prop.~\ref{OProp2}};
\draw (6.5,3.9) node[above]{Thm.~\ref{MT2}};

\draw (2,2.6) node[above]{Thm.~\ref{MT4}};
\draw (6.5,2.6) node[above]{Thm.~\ref{MT4}};

\draw (2,1.3) node[above]{Thm.~\ref{MT3}};
\draw (6.5,1.3) node[above]{Thm.~\ref{MT3}};

\end{tikzpicture}
\caption{Implications for models of $\PA^-$}\label{Fig2}
\end{figure}

As applications of our results, for several theories $T$, we show that every $\Delta_0$-elementary cofinal extension of a model of $T$ is also a model of $T$. 
For example, as an easy consequence of Proposition \ref{GA0}, we have the following corollary:

\begin{cor}\label{preserve}
   Let $M, K \models \PA^-$ be such that $M \dcf K$. 
   \begin{enumerate}
       \item If $M \models \mathbf{I}\Delta_0$, then $K \models \mathbf{I}\Delta_0$. 
       \item If $M \models \mathbf{I}\Pi_1^-$, then $K \models \mathbf{I}\Pi_1^-$.
   \end{enumerate}
\end{cor}
\begin{proof}
By Proposition \ref{GA0}, we have $M \prec_{\Sigma_1} K$. 
Then, these clauses follow from the facts that $\mathbf{I}\Delta_0$ is axiomatized by a set of $\Pi_1$ sentences and that $\mathbf{I}\Pi_1^-$ is axiomatized by a set of $\Sigma_2$ sentences (Cf.~\cite[Proposition 3.1]{KPD}). 
\end{proof}

David Belanger and Tin Lok Wong independently proved the following theorem concerning studies in this direction.\footnote{This result was informed by Wong through private communication.}

\begin{thm}[Belanger and Wong]\label{BW}
   Let $M, K \models \PA^-$ be such that $M \dcf K$. 
   \begin{enumerate}
       \item If $M \models \mathbf{I}\Sigma_{n+1}$, then $K \models \mathbf{I}\Sigma_{n+1}$. 
       \item If $M \models \mathbf{B}\Sigma_{n+1} + \exp$, then $K \models \mathbf{B}\Sigma_{n+1} + \exp$.
   \end{enumerate}
\end{thm}

We will show that an analogous preservation property also holds for the theories $\Coll(\Sigma_{n+1})$, $\CWm(\Sigma_{n+1})$, $\mathbf{B}\Sigma_{n+1}^-$ and $\mathbf{I}\Sigma_{n+1}^-$ (See Theorem \ref{Appl1} and Corollary \ref{Cor_Appl2}).

\section{Strong collection schemes}\label{Sec:SC}

In this section, from the viewpoint of Gaifman's splitting theorem, we prove a theorem on the model theoretic characterization of $\CS(\Sigma_{n+1})$. 
As consequences of the result, we refine several already known results such as Theorems \ref{Kaye1} and \ref{Mijajlovic1}. 
Also, as an application of the result, we prove that every $\Delta_0$-elementary cofinal extension of a model of $\PA^- + \CS(\Sigma_{n+1})$ is also a model of $\CS(\Sigma_{n+1})$. 

\begin{thm}\label{MT1}
For any $M \models \PA^-$ and $n \geq 0$, the following are equivalent: 
\begin{enumerate}
 \item $M$ satisfies the condition $\eex{n+1}$. 
 That is, for any $N \models \PA^-$, if $M \prec N$, then $\sup_N(M) \prec_{\Sigma_{n+1}} N$. 

 \item $M \models \CS(\Sigma_{n+1})$. 

 \item \textup{($n \geq 1$):} $M \models \Coll(\Sigma_n)$ and for any $N \models \PA^- + \Coll(\Sigma_n)$, if $M \prec_{\Sigma_{n+1}} N$, then $\sup_N(M) \prec_{\Sigma_{n+1}} N$. \\
    \textup{($n = 0$):} For any $N \models \PA^-$, if $M \prec_{\Sigma_1} N$, then $\sup_N(M) \prec_{\Sigma_1} N$. 
    
 \item For any $N \models \PA^-$, if $M \prec_{\Sigma_{n+2}} N$, then $\sup_N(M) \prec_{\Sigma_{n+1}} N$. 

\end{enumerate}
\end{thm}
\begin{proof}
$(1 \Rightarrow 2)$: 
Suppose that $M$ satisfies the condition $\eex{n+1}$. 
It suffices to show that $M \models \CSm(\Sigma_{n+1})$ by Proposition \ref{OProp3}. 
Assume, towards a contradiction, that $M \not \models \CSm(\Sigma_{n+1})$. 
We then obtain some $\Sigma_{n+1}$ formula $\varphi(\vec{x}, \vec{y})$ and $\vec{a} \in M$ such that 
\begin{align}\label{CS}
	M \models \forall \vec{v}\, \exists \vec{x} < \vec{a}\, \bigl(\exists \vec{y}\, \varphi(\vec{x}, \vec{y}) \land \forall \vec{y} < \vec{v}\, \neg \varphi(\vec{x}, \vec{y}) \bigr). 
\end{align}
We prepare new constant symbols $\vec{\mathsf{c}}$. 
For each tuple $\vec{b} \in M$, let $T_{\vec{b}}$ be the $\LA \cup M \cup \{\vec{\mathsf{c}}\}$-theory defined by:
\[
	T_{\vec{b}} : = \ElemDiag(M) \cup \{\vec{\mathsf{c}} < \vec{a}\} \cup \{\exists \vec{y}\, \varphi(\vec{\mathsf{c}}, \vec{y}) \} \cup \{\forall \vec{y} < \vec{b}\, \neg \varphi(\vec{\mathsf{c}}, \vec{y})\}. 
\]
For such $\vec{b}$, by (\ref{CS}), we find $\vec{e} < \vec{a}$ such that 
\[
	M \models \exists \vec{y}\, \varphi(\vec{e}, \vec{y}) \land \forall \vec{y} < \vec{b}\, \neg \varphi(\vec{e}, \vec{y}). 
\]
This gives a model of $T_{\vec{b}}$ by taking $\vec{e}$ as the interpretation of the constant symbols $\vec{\mathsf{c}}$.
Hence, by the compactness theorem, we obtain a model of the theory
\[
	\ElemDiag(M) \cup \{\vec{\mathsf{c}} < \vec{a}\} \cup \{\exists \vec{y}\, \varphi(\vec{\mathsf{c}}, \vec{y}) \} \cup \{\forall \vec{y} < \vec{b}\, \neg \varphi(\vec{\mathsf{c}}, \vec{y}) \mid \vec{b} \in M\}, 
\]
and let $N$ be the restriction of the model to the language $\LA$. 
Then, $M \prec N$, and so we obtain that $\sup_N(M) \prec_{\Sigma_{n+1}} N$ by the condition $\eex{n+1}$.  
Since $N \models \vec{\mathsf{c}}^N < \vec{a}$, we have $\vec{\mathsf{c}}^N \in \sup_N(M)$. 
We then obtain $\sup_N(M) \models \exists \vec{y}\, \varphi(\vec{\mathsf{c}}^N, \vec{y})$ because $N \models \exists \vec{y}\, \varphi(\vec{\mathsf{c}}^N, \vec{y})$ and $\sup_N(M) \prec_{\Sigma_{n+1}} N$. 
On the other hand, since $N \models \forall \vec{y} < \vec{b}\, \neg \varphi(\vec{\mathsf{c}}^N, \vec{y})$ for all $\vec{b} \in M$, we obtain $N \models \neg \varphi(\vec{\mathsf{c}}^N, \vec{k})$ for all $\vec{k} \in \sup_N(M)$ because $M \cf \sup_N(M)$. 
Hence, $\sup_N(M) \models \forall \vec{y} \, \neg \varphi(\vec{\mathsf{c}}^N, \vec{y})$ because $\sup_N(M) \prec_{\Sigma_{n+1}} N$ again. 
This is a contradiction. 
%We have proved that $M$ is a model of $\CS(\Sigma_{n+1})$.

$(2 \Rightarrow 3)$: Suppose that $M$ is a model of $\CS(\Sigma_{n+1})$. 
Then, $M \models \Coll(\Sigma_n)$ by Proposition \ref{OProp1}. 
Let $N$ be any model of $\PA^-$ with $M \prec_{\Sigma_{n+1}} N$. 
In the case of $n \geq 1$, we further assume $N \models \Coll(\Sigma_n)$. 
We would like to show $\sup_N(M) \prec_{\Sigma_{n+1}} N$. 
By the Tarski--Vaught test (cf.~\cite[Exercise 7.4]{Kaye}), it suffices to show that for any $\Pi_n$ formula $\varphi(\vec{x}, \vec{y})$ and any $\vec{a} \in \sup_N(M)$, if $N \models \exists \vec{y}\, \varphi(\vec{a}, \vec{y})$, then $N \models \varphi(\vec{a}, \vec{d})$ for some $\vec{d} \in \sup_N(M)$. 

Suppose that $N \models \exists \vec{y}\, \varphi(\vec{a}, \vec{y})$ for some $\Pi_n$ formula $\varphi(\vec{x}, \vec{y})$ and $\vec{a} \in \sup_N(M)$. 
We find some $\vec{b} \in M$ such that $\sup_N(M) \models \vec{a} < \vec{b}$. 
Since $M \models \CS(\Sigma_{n+1})$, we have 
\[
	M \models \forall \vec{u}\, \exists \vec{v}\, \forall \vec{x} < \vec{u}\, \bigl(\exists \vec{y}\, \varphi(\vec{x}, \vec{y}) \to \exists \vec{y} < \vec{v}\, \varphi(\vec{x}, \vec{y}) \bigr). 
\]
So, we find $\vec{c} \in M$ such that
\[
	M \models \forall \vec{x} < \vec{b}\, \bigl(\exists \vec{y}\, \varphi(\vec{x}, \vec{y}) \to \exists \vec{y} < \vec{c}\, \varphi(\vec{x}, \vec{y}) \bigr). 
\]
In the case of $n = 0$, the formula $\exists \vec{y} < \vec{w}\, \varphi(\vec{x}, \vec{y})$ is a $\Delta_0$ formula. 
In the case of $n \geq 1$, the formula $\exists \vec{y} < \vec{w}\, \varphi(\vec{x}, \vec{y})$ may be treated as a $\Pi_n$ formula in both $M$ and $N$ because they are models of $\PA^- + \Coll(\Sigma_n)$. 
    So, in either case, the formula $\forall \vec{x} < \vec{z}\, \bigl(\exists \vec{y}\, \varphi(\vec{x}, \vec{y}) \to \exists \vec{y} < \vec{w}\, \varphi(\vec{x}, \vec{y}) \bigr)$ is equivalent to a $\Pi_{n+1}$ formula. 
    Therefore, it follows from $M \prec_{\Sigma_{n+1}} N$ that
\[
	N \models \forall \vec{x} < \vec{b}\, \bigl(\exists \vec{y}\, \varphi(\vec{x}, \vec{y}) \to \exists \vec{y} < \vec{c}\, \varphi(\vec{x}, \vec{y}) \bigr). 
\]
Since $N \models \vec{a} < \vec{b}$ and $N \models \exists \vec{y}\, \varphi(\vec{a}, \vec{y})$, we obtain $N \models \exists \vec{y} < \vec{c}\, \varphi(\vec{a}, \vec{y})$. 
Hence, we find some $\vec{d} < \vec{c}$ such that $N \models \varphi(\vec{a}, \vec{d})$. 
Then, $\vec{d} \in \sup_N(M)$. 
This completes the proof. 

$(3 \Rightarrow 4)$: 
    Suppose that $M$ satisfies the condition stated in Clause 3. 
    In the case of $n = 0$, we are done. 
    So, we may assume $n \geq 1$. 
    Let $N \models \PA^-$ be such that $M \prec_{\Sigma_{n+2}} N$. 
    Since $M \models \Coll(\Sigma_n)$ and the theory $\PA^- + \Coll(\Sigma_n)$ is axiomatized by a set of $\Pi_{n+2}$ sentences, we have $N \models \Coll(\Sigma_n)$. 
    So, we conclude $\sup_N(M) \prec_{\Sigma_{n+1}} N$ by Clause 3. 

$(4 \Rightarrow 1)$: Trivial. 
\end{proof}

We immediately obtain the following corollary: 

\begin{cor}\label{CorMT1}
For any $M \models \mathbf{I}\Delta_{0}$ and $n \geq 0$, $M \models \mathbf{I}\Sigma_{n}$ if and only if $M$ satisfies the condition $\eex{n}$. 
\end{cor}
\begin{proof}
The equivalence for $n = 0$ is trivial. 
The equivalence for $n \geq 1$ follows from Theorem \ref{MT1} because $\mathbf{I}\Delta_{0} + \CS(\Sigma_{n+1})$ is equivalent to $\mathbf{I}\Sigma_{n+1}$. 
\end{proof}

The following refinement of Mijajlovi{\'c}'s and Kaye's theorems (Theorems \ref{Mijajlovic1} and \ref{Kaye3}) follows from Theorem \ref{MT1}. 

\begin{cor}Let $M, N \models \PA^-$ and $n \geq 1$. 
\begin{enumerate}
    \item If $M \models \CS(\Sigma_{1})$ and $M \prec_{\Sigma_{1}} N$, then $\sup_N(M) \prec_{\Sigma_{1}} N$. 
    
    \item If $M \models \CS(\Sigma_{n+1})$, $N \models \Coll(\Sigma_n)$, and $M \prec_{\Sigma_{n+1}} N$, then $\sup_N(M) \prec_{\Sigma_{n+1}} N$. 
\end{enumerate}
\end{cor}

As an application of Theorem \ref{MT1}, we prove the following theorem whose first clause is a refinement of Theorem \ref{Kaye1} and whose second clause is a refinement of the first clause of Theorem \ref{BW}. 

\begin{thm}\label{Appl1}
Let $n \geq 0$ and $M, K \models \PA^-$ be such that $M \dcf K$. 
\begin{enumerate}
    \item If $M \models \Coll(\Sigma_{n+1})$, then $M \prec_{\Sigma_{n+2}} K$. 
    \item If $M \models \CS(\Sigma_{n+1})$, then $K \models \CS(\Sigma_{n+1})$. 
\end{enumerate}
\end{thm}
\begin{proof}
The case of $n = 0$ for Clause 1 is exactly Clause 1 of Theorem \ref{Kaye1}, and so we are done. 
We simultaneously prove the following two statements by induction on $n$. 
\begin{enumerate}
    \item If $M \models \Coll(\Sigma_{n+2})$, then $M \prec_{\Sigma_{n+3}} K$. 
    \item If $M \models \CS(\Sigma_{n+1})$, then $K \models \CS(\Sigma_{n+1})$. 
\end{enumerate}
We assume that these statements hold for all $k < n$. 

Firstly, we prove Clause 2. 
Suppose $M \models \CS(\Sigma_{n+1})$. 
We have $M \models \Coll(\Sigma_{n+1})$ by Proposition \ref{OProp1}. 
By the induction hypothesis for Clause 1, we obtain $M \prec_{\Sigma_{n+2}} K$. 
Let $N \models \PA^-$ be any model such that $K \prec N$. 
Then, we have $M \prec_{\Sigma_{n+2}} N$. 
Since $M \models \CS(\Sigma_{n+1})$, by Theorem \ref{MT1}, we obtain $\sup_N(M) \prec_{\Sigma_{n+1}} N$. 
Since $M \cf K$, we get $\sup_N(M) = \sup_N(K)$, and thus $\sup_N(K) \prec_{\Sigma_{n+1}} N$. 
We have proved that $K$ satisfies the condition $\eex{n+1}$. 
By Theorem \ref{MT1} again, we obtain $K \models \CS(\Sigma_{n+1})$. 
\footnote{The following direct argument of this part, which does not use Theorem \ref{MT1}, is due to Tin Lok Wong: Suppose $M \models \CS(\Sigma_{n+1})$. 
Let $\vec{a} \in K$ be any elements and $\varphi(\vec{x}, \vec{y})$ be any $\Sigma_{n+1}$ formula. 
Since $M \cf K$, we find $\vec{b} \in M$ such that $\vec{a} < \vec{b}$. 
Then, for some $\vec{c} \in M$, we have $M \models \forall \vec{x} < \vec{b}\, \bigl(\exists \vec{y}\, \varphi(\vec{x}, \vec{y}) \to \exists \vec{y} < \vec{c}\, \varphi(\vec{x}, \vec{y}) \bigr)$. 
This formula is logically equivalent to some $\Pi_{n+2}$ formula, and so it is also true in $K$ because $M \prec_{\Sigma_{n+2}} K$ by the induction hypothesis. 
In particular, $K \models \forall \vec{x} < \vec{a}\, \bigl(\exists \vec{y}\, \varphi(\vec{x}, \vec{y}) \to \exists \vec{y} < \vec{c}\, \varphi(\vec{x}, \vec{y}) \bigr)$. 
We have shown that $K \models \CSm(\Sigma_{n+1})$. 
By Proposition \ref{OProp3}, we conclude $K \models \CS(\Sigma_{n+1})$.}

Secondly, we prove Clause 1. 
Suppose $M \models \Coll(\Sigma_{n+2})$. 
Then, $M \models \CS(\Sigma_{n+1})$ by Proposition \ref{w_to_s}. 
We have already proved that $K \models \CS(\Sigma_{n+1})$ in Clause 2. 
Let $\vec{a} \in M$ and $\varphi(\vec{x}, \vec{y}, \vec{w})$ be any $\Sigma_{n+1}$ formula such that $K \models \exists \vec{x}\, \forall \vec{y}\, \varphi(\vec{x}, \vec{y}, \vec{a})$. 
Since $M \cf K$, there exist $\vec{b} \in M$ such that for all $\vec{c} \in M$, we have $K \models \exists \vec{x} < \vec{b}\, \forall \vec{y} < \vec{c}\, \varphi(\vec{x}, \vec{y}, \vec{a})$. 
Since both $M$ and $K$ are models of $\Coll(\Sigma_{n+1})$, the above formula can be treated as a $\Sigma_{n+1}$ formula. 
By the induction hypothesis, we have $M \prec_{\Sigma_{n+2}} K$. 
Then, $M \models \exists \vec{x} < \vec{b}\, \forall \vec{y} < \vec{c}\, \varphi(\vec{x}, \vec{y}, \vec{a})$ and hence $M \models \forall \vec{v}\, \exists \vec{x} < \vec{b}\, \forall \vec{y} < \vec{v}\, \varphi(\vec{x}, \vec{y}, \vec{a})$. 
By applying $\Coll(\Sigma_{n+2})$, we get $M \models \exists \vec{x} < \vec{b}\, \forall \vec{y}\, \varphi(\vec{x}, \vec{y}, \vec{a})$. 
So, we conclude $M \models \exists \vec{x} \, \forall \vec{y}\, \varphi(\vec{x}, \vec{y}, \vec{a})$. 
\end{proof}

% It is known that the theory $\mathbf{I}\Pi_{n+2}^-$ is axiomatized by some set of $\Sigma_{n+3}$ sentences (cf.~\cite[Proposition 3.1]{KPD}). 
% Then, we obtain the following corollary immediately from Clause 1 of Theorem \ref{Appl1}. 

% \begin{cor}
% Let $n \geq 0$ and $M, K \models \PA^-$ be such that $M \dcf K$. 
% If $M \models \mathbf{I}\Pi_{n+2}^- + \Coll(\Sigma_{n+1})$, then $K \models \mathbf{I}\Pi_{n+2}^-$. 
% \end{cor}

\section{Collection schemes}\label{Sec:Coll}

It follows from Theorem \ref{Appl1} that every model of $\PA^- + \Coll(\Sigma_{n+1})$ satisfies the condition $\COF{n+2}$. 
Continuing this line of observation, we prove the following theorem on the model theoretic characterization of $\Coll(\Sigma_{n+1})$. 

\begin{thm}\label{MT2}
For any $M \models \PA^-$ and $n \geq 0$, the following are equivalent: 
\begin{enumerate}
     \item $M \models \Coll(\Sigma_{n+1})$. 

    \item $M$ satisfies the conditions $\eex{n}$ and $\COF{n+2}$. 

	\item $M$ satisfies the conditions $\eex{n}$ and $\cof{n+2}$.  

\end{enumerate}
\end{thm}
\begin{proof}
$(1 \Rightarrow 2)$: Suppose $M \models \Coll(\Sigma_{n+1})$. 
By Theorem \ref{Appl1}, $M$ satisfies $\COF{n+2}$. 
It suffices to prove that $M$ satisfies $\eex{n}$. 
The case $n = 0$ is trivial, and so we may assume $n > 0$. 
By Proposition \ref{w_to_s}, we have $M \models \CS(\Sigma_n)$. 
It follows from Theorem \ref{MT1} that $M$ satisfies $\eex{n}$. 

$(2 \Rightarrow 3)$: Trivial. 

$(3 \Rightarrow 1)$: Suppose that $M$ satisfies the conditions $\eex{n}$ and $\cof{n+2}$. 
We prove that the contrapositive of each instance of $\Coll(\Sigma_{n+1})$ holds in $M$. 
For any $\Sigma_{n+1}$ formula $\varphi(\vec{x}, \vec{y}, \vec{z})$ and $\vec{a}, \vec{b} \in M$, we assume
\[
	M \models \forall \vec{v}\, \exists \vec{x} < \vec{a}\, \forall \vec{y} < \vec{v}\, \neg \varphi(\vec{x}, \vec{y}, \vec{b}). 
\]
By the compactness argument, we obtain an $N \models \PA^-$ such that $M \prec N$ and $\sup_N(M) \neq N$. 
We fix some $\vec{e} \in N \setminus \sup_N(M)$. 
Since $M \prec N$, we have
\[
	N \models \forall \vec{v}\, \exists \vec{x} < \vec{a}\, \forall \vec{y} < \vec{v}\, \neg \varphi(\vec{x}, \vec{y}, \vec{b}),
\]
and so we find some $\vec{c} \in N$ with $\vec{c} < \vec{a}$ such that for all $\vec{d} \in N$ with $\vec{d} < \vec{e}$, we have $N \models \neg \varphi(\vec{c}, \vec{d}, \vec{b})$. 
Then, $\vec{c} \in \sup_N(M)$. 
Also, since $\vec{e} \in N \setminus \sup_N(M)$ and $\sup_N(M) \ee N$, we get that $N \models \neg \varphi(\vec{c}, \vec{d}, \vec{b})$ holds for all $\vec{d} \in \sup_N(M)$. 
Since $\neg \varphi$ is a $\Pi_{n+1}$ formula and $\sup_N(M) \prec_{\Sigma_{n}} N$ holds by the condition $\eex{n}$, we obtain $\sup_N(M) \models \neg \varphi(\vec{c}, \vec{d}, \vec{b})$. 
Thus, $\sup_N(M) \models \exists \vec{x} < \vec{a}\, \forall \vec{y}\, \neg \varphi(\vec{x}, \vec{y}, \vec{b})$. 
By the condition $\cof{n+2}$, we have $M \prec_{\Sigma_{n+2}} \sup_N(M)$. 
Thus, we conclude $M \models \exists \vec{x} < \vec{a}\, \forall \vec{y}\, \neg \varphi(\vec{x}, \vec{y}, \vec{b})$. 
We have proved that $M \models \Coll(\Sigma_{n+1})$. 
\end{proof}

We immediately obtain the following corollary: 

\begin{cor}\label{cofCOF}
    For any $n \geq 0$ and $M \models \PA^-$ satisfying $\eex{n}$, $M$ satisfies $\cof{n+2}$ if and only if $M$ satisfies $\COF{n+2}$. 
\end{cor}

By combining Corollary \ref{cofCOF} and Proposition \ref{GA0}, we obtain the following refinement of Proposition \ref{eexcof}. 

\begin{cor}\label{eexcof2}
For any $n \geq 0$ and $M \models \PA^-$, if $M$ satisfies $\eex{n}$, then $M$ also satisfies $\COF{n+1}$. 
\end{cor}

\begin{rem}\label{RemMT2}
For models $M$ of $\mathbf{I}\Delta_0$, the implication $(3 \Rightarrow 1)$ of Theorem \ref{MT2} is immediately proved by using Theorem \ref{Clote} and Corollary \ref{CorMT1}. 
For, suppose $M \models \mathbf{I}\Delta_0$ and $M$ satisfies the conditions $\eex{n}$ and $\cof{n+2}$. 
By Corollary \ref{CorMT1}, $M$ is a model of $\mathbf{I}\Sigma_n$. 
We can easily find an $N \models \mathbf{I}\Sigma_n$ such that $M \prec N$ and $\sup_N(M) \neq N$ by using the compactness theorem. 
We have $M \prec_{\Sigma_{n+2}} \sup_N(M) \prec_{\Sigma_n} N$ by the conditions $\eex{n}$ and $\cof{n+2}$. 
By Theorem \ref{Clote}, we have $\sup_N(M) \models \mathbf{B}\Sigma_{n+1}$, and hence $M \models \mathbf{B}\Sigma_{n+1}$ because $\mathbf{B}\Sigma_{n+1}$ is axiomatized by a set of $\Pi_{n+3}$ sentences (cf.~\cite[Exercise 10.2.(a)]{Kaye}). 
\end{rem}

By refining the argument presented in Remark \ref{RemMT2}, we show that for models of $\mathbf{I}\Delta_0$ satisfying $\eex{n}$, the condition $\cof{n+2}$ is equivalent to some weaker conditions. 

\begin{prop}\label{cof_prop2}
Let $n \geq 0$. 
If $M \models \mathbf{I}\Delta_0$ satisfies $\eex{n}$, then the following are equivalent: 
\begin{enumerate}
    \item $M$ satisfies $\cof{n+2}$. 
    \item For any $N \models \PA^-$, if $M \prec N$, then there exists an $N' \models \PA^-$ such that $N \prec N'$ and $M \prec_{\Sigma_{n+2}} \sup_{N'}(M)$. 
    \item There exists an $N \models \PA^-$ such that $M \prec N$, $N \neq \sup_{N}(M)$ and $M \prec_{\Sigma_{n+2}} \sup_{N}(M)$. 
\end{enumerate}
\end{prop}
\begin{proof}
Let $M \models \mathbf{I}\Delta_0$ satisfy $\eex{n}$. 
By Corollary \ref{CorMT1}, we have $M \models \mathbf{I}\Sigma_{n}$. 

$(1 \Rightarrow 2)$: Obvious by letting $N' = N$. 

$(2 \Rightarrow 3)$: Suppose that $M$ satisfies the condition of Clause 2. 
By the compactness argument, we find some $N \models \PA^-$ such that $M \prec N$ and $\sup_N(M) \neq N$. 
By Clause 2, we also find some $N' \models \PA^-$ such that $N \prec N'$ and $M \prec_{\Sigma_{n+2}} \sup_{N'}(M)$. 
We have $M \prec N'$, $N' \neq \sup_{N'}(M)$ and $M \prec_{\Sigma_{n+2}} \sup_{N'}(M)$. 

$(3 \Rightarrow 1)$: Let $N \models \PA^-$ be such that $M \prec N$, $N \neq \sup_N(M)$ and $M \prec_{\Sigma_{n+2}} \sup_N(M)$. 
By Theorem \ref{Kaye3}, $N$ is a proper $\Sigma_{n}$-elementary extension of $\sup_N(M)$. 
Since $N \models \mathbf{I}\Sigma_n$, by Theorem \ref{Clote}, we have $\sup_{N}(M) \models \mathbf{B}\Sigma_{n+1}$. 
As in the argument in Remark \ref{RemMT2}, we get $M \models \mathbf{B}\Sigma_{n+1}$. 
By Theorem \ref{MT2}, $M$ satisfies the condition $\cof{n+2}$. 
\end{proof}

The second condition in Proposition \ref{cof_prop2} originates from Kaye \cite{Kaye91}.

\begin{prob}
    Can the assumption $M \models \mathbf{I}\Delta_0$ in Proposition \ref{cof_prop2} be weakened to $M \models \PA^-$?
\end{prob}

As a straightforward consequence of Theorems \ref{MT1} and \ref{MT2}, we obtain the following refinement of Theorem \ref{Kaye2}. 

\begin{cor}
For $M \models \PA^-$, the following are equivalent: 
\begin{enumerate}
    \item $M \models \bigcup_{n \in \omega} \Coll(\Sigma_n)$.
    \item For any $N \models \PA^-$, if $M \prec N$, then $M \prec \sup_N(M) \prec N$. 
    \item For any $N \models \PA^-$, if $M \prec N$, then $\sup_N(M) \prec N$. 
\end{enumerate}
\end{cor}
\begin{proof}
$(1 \Rightarrow 2)$: Suppose $M \models \bigcup_{n \in \omega} \Coll(\Sigma_n)$. 
By Theorem \ref{MT2}, $M$ satisfies $\eex{n}$ and $\cof{n}$ for all $n \in \omega$. 
This means that $M$ satisfies Clause 2. 

$(2 \Rightarrow 3)$: Trivial. 

$(3 \Rightarrow 1)$: Suppose that $M$ satisfies Clause 3. 
Then, $M$ satisfies $\eex{n}$ for all $n \in \omega$. 
By Theorem \ref{MT1}, $M \models \CS(\Sigma_n)$ for all $n \in \omega$. 
Then, $M \models \bigcup_{n \in \omega} \Coll(\Sigma_n)$ by Proposition \ref{OProp1}. 
\end{proof}

We propose the following problem. 

\begin{prob}\label{preserveProb}
Let $n \geq 0$ and $M, K \models \PA^-$ be such that $M \dcf K$. 
\begin{enumerate}
    \item Does $M \models \Coll(\Sigma_{n+1})$ imply $K \models \Coll(\Sigma_{n+1})$?
    \item If $M$ satisfies $\cof{n+1}$, then does $K$ satisfy $\cof{n+1}$?
\end{enumerate}
\end{prob}

Belanger and Wong's Theorem \ref{BW} provides the affirmative answer to the first clause of Problem \ref{preserveProb} in the case of $M \models \mathbf{I}\Delta_0 + \exp$. 
By Theorem \ref{MT2}, we have that $M \models \Coll(\Sigma_1)$ if and only if $M$ satisfies $\cof{2}$. 
So, Theorem \ref{BW} also provides the affirmative answer to the second clause of Problem \ref{preserveProb} in the case of $M \models \mathbf{I}\Delta_0 + \exp$ and $n = 1$.

\section{Weak parameter-free collection schemes}\label{Sec:WColl-}

In this subsection, we prove a model theoretic characterization of the weak parameter-free collection scheme $\CWm(\Sigma_{n+1})$. 
As a consequence, we show that every $\Delta_0$-elementary cofinal extension of a model of one of the theories $\PA^- + \CWm(\Sigma_{n+1})$, $\mathbf{B}\Sigma_{n+1}^-$ and $\mathbf{I}\Sigma_{n+1}^-$ is also a model of the theory. 

\begin{defn}
Let $M, K \models \PA^-$ be such that $M \subseteq K$ and let $n \geq 0$.  
We write $M \equiv_{\Sigma_{n}} K$ iff $M$ and $K$ satisfy the same $\Sigma_n$ sentences. \end{defn}

We introduce the following weak variations of the conditions $\eex{n}$, $\cof{n}$ and $\COF{n}$.   

\begin{defn}\label{Def_equiv}
Let $M \models \PA^-$ and $n \geq 0$. 
\begin{itemize}
		\item We say that $M$ satisfies the condition $\eexm{n}$ iff for any $N \models \PA^-$, if $M \prec N$, then $\sup_N(M) \equiv_{\Sigma_{n}} N$. 

\item We say that $M$ satisfies the condition $\cofm{n}$ iff for any $N \models \PA^-$, if $M \prec N$, then $M \equiv_{\Sigma_{n}} \sup_N(M)$. 

\item We say that $M$ satisfies the condition $\COFm{n}$ iff for any $K \models \PA^-$, if $M \dcf K$, then $M \equiv_{\Sigma_{n}} K$. 
\end{itemize}
\end{defn}

For any models $M, N \models \PA^-$ with $M \prec N$, it is easy to see that $M \equiv_{\Sigma_n} \sup_N(M)$ if and only if $\sup_N(M) \equiv_{\Sigma_n} N$. 
So, we have the following proposition and we may focus only on the conditions $\cofm{n}$ and $\COFm{n}$:

\begin{prop}
For any $M \models \PA^-$ and $n \geq 0$, $M$ satisfies $\eexm{n}$ if and only if $M$ satisfies $\cofm{n}$. 
\end{prop}

Theorem \ref{MT2} states that the combination of the conditions $\eex{n}$ and $\cof{n+1}$ characterizes $\Coll(\Sigma_{n+1})$. 
If $\cof{n+1}$ is weakened to $\cofm{n+1}$, then we obtain the following characterization of $\CWm(\Sigma_{n+1})$. 

\begin{thm}\label{MT3}
For any $M \models \PA^-$ and $n \geq 0$, the following are equivalent: 
\begin{enumerate}
	\item $M \models \CWd(\Sigma_{n+1})$. 

 \item $M \models \CWm(\Sigma_{n+1})$. 

 \item $M$ satisfies the conditions $\eex{n}$ and $\COFm{n+2}$. 

	\item $M$ satisfies the conditions $\eex{n}$ and $\cofm{n+2}$. 
\end{enumerate}
\end{thm}
\begin{proof}
$(1 \Rightarrow 2)$: Trivial. 

$(2 \Rightarrow 3)$: Suppose $M \models \CWm(\Sigma_{n+1})$. 
We show that $M$ satisfies $\eex{n}$. 
If $n = 0$, $M$ trivially satisfies $\eex{0}$. 
If $n \geq 1$, by Proposition \ref{w_to_s}, $M \models \CS(\Sigma_{n})$. 
By Theorem \ref{MT1}, $M$ satisfies $\eex{n}$. 

We prove that $M$ satisfies $\COFm{n+2}$. 
Let $K \models \PA^-$ be such that $M \dcf K$, and we show $M \equiv_{\Sigma_{n+2}} K$. 
By Corollary \ref{eexcof2}, we have that $M$ satisfies $\COF{n+1}$, and thus $M \prec_{\Sigma_{n+1}} K$ holds. 
Then, it suffices to show that $K \models \psi$ implies $M \models \psi$ for all $\Sigma_{n+2}$ sentences $\psi$.

Let $\varphi(\vec{x}, \vec{y})$ be any $\Sigma_{n}$ formula such that $K \models \exists \vec{x}\, \forall \vec{y}\, \varphi(\vec{x}, \vec{y})$. 
Since $M \cf K$, there exist $\vec{a} \in M$ such that for all $\vec{b} \in M$, $K \models \exists \vec{x} < \vec{a}\, \forall \vec{y} < \vec{b}\, \varphi(\vec{x}, \vec{y})$. 
If $n = 0$, this formula is $\Delta_0$.
If $n \geq 1$, since $M \models \CS(\Sigma_n)$, we have that $K \models \CS(\Sigma_n)$ by Theorem \ref{Appl1}. 
In particular, both $M$ and $K$ are models of $\Coll(\Sigma_n)$, and hence that formula above may be regarded as $\Sigma_n$ in $M$ and $K$. 
Thus, we have $M \models \exists \vec{x} < \vec{a}\, \forall \vec{y} < \vec{b}\, \varphi(\vec{x}, \vec{y})$ because $M \prec_{\Sigma_{n+1}} K$. 
Therefore, $\exists \vec{u}\, \forall \vec{v}\, \exists \vec{x} < \vec{u}\, \forall \vec{y} < \vec{v}\, \varphi(\vec{x}, \vec{y})$ is true in $M$. 
By applying $\CWm(\Sigma_{n+1})$, we obtain $M \models \exists \vec{x}\, \forall \vec{y}\, \varphi(\vec{x}, \vec{y})$. 

$(3 \Rightarrow 4)$: Trivial. 

$(4 \Rightarrow 1)$: Suppose that $M$ satisfies the conditions $\eex{n}$ and $\cofm{n+2}$. 
We prove that the contrapositive of each instance of $\CWd(\Sigma_{n+1})$ holds in $M$. 
For any $\Sigma_{n+1}$ formula $\varphi(\vec{x}, \vec{y}, \vec{z})$, we assume
\[
	M \models \exists \vec{z}\, \exists \vec{u}\, \forall \vec{v}\, \exists \vec{x} < \vec{u}\, \forall \vec{y} < \vec{v}\, \neg \varphi(\vec{x}, \vec{y}, \vec{z}). 
\]
Then, we find $\vec{a}, \vec{b} \in M$ such that $M \models \forall \vec{v}\, \exists \vec{x} < \vec{a}\, \forall \vec{y} < \vec{v}\, \neg \varphi(\vec{x}, \vec{y}, \vec{b})$. 
By the same argument as in the proof of Theorem \ref{MT2}, we obtain that $\sup_N(M) \models \exists \vec{x} < \vec{a}\, \forall \vec{y}\, \neg \varphi(\vec{x}, \vec{y}, \vec{b})$ by using the condition $\eex{n}$. 
So, we have $\sup_N(M) \models \exists \vec{z}\, \exists \vec{x}\, \forall \vec{y}\, \neg \varphi(\vec{x}, \vec{y}, \vec{z})$. 
Since $M \equiv_{\Sigma_{n+2}} \sup_N(M)$ by the condition $\cofm{n+2}$, we conclude $M \models \exists \vec{z}\, \exists \vec{x}\, \forall \vec{y}\, \neg \varphi(\vec{x}, \vec{y}, \vec{z})$.
We have proved that $M \models \CWd(\Sigma_{n+1})$. 

\end{proof}

It is known that each of $\PA^- + \CWm(\Sigma_{n+1})$ and the extensions $\mathbf{I}\Sigma_{n+1}^-$ and $\mathbf{B}\Sigma_{n+1}^-$ of $\PA^- + \CWm(\Sigma_{n+1})$ are axiomatized by some set of Boolean combinations of $\Sigma_{n+2}$ sentences (cf.~\cite[Propositions 3.2 and 3.3]{KPD}). 
Hence, we obtain the following corollary.

\begin{cor}\label{Cor_Appl2}
Let $n \geq 0$ and $M, K \models \PA^-$ be such that $M \dcf K$. 
\begin{enumerate}
    \item If $M \models \CWm(\Sigma_{n+1})$, then $K \models \CWm(\Sigma_{n+1})$.
    \item If $M \models \mathbf{B}\Sigma_{n+1}^-$, then $K \models \mathbf{B}\Sigma_{n+1}^-$.
    \item If $M \models \mathbf{I}\Sigma_{n+1}^-$, then $K \models \mathbf{I}\Sigma_{n+1}^-$.
    % \item If $M \models \mathbf{I}\Pi_{n+2}^-$, then $K \models \mathbf{I}\Pi_{n+2}^-$.
\end{enumerate}
\end{cor}

In the case of $M \models \mathbf{I}\Delta_0$, the following proposition is proved in the similar way as in the proof of Proposition \ref{cof_prop2} by using Theorem \ref{MT3} and the fact that $\mathbf{B}\Sigma_{n+1}^-$ is axiomatized by some set of Boolean combinations of $\Sigma_{n+2}$ sentences. 

\begin{prop}\label{cofm_prop}
Let $n \geq 0$. 
If $M \models \mathbf{I}\Delta_0$ satisfies $\eex{n}$, then the following are equivalent: 
\begin{enumerate}
    \item $M$ satisfies $\cofm{n+2}$. 
    \item For any $N \models \PA^-$, if $M \prec N$, then there exists an $N' \models \PA^-$ such that $N \prec N'$ and $M \equiv_{\Sigma_{n+2}} \sup_{N'}(M)$. 
    \item There exists an $N \models \PA^-$ such that $M \prec N$, $N \neq \sup_{N}(M)$ and $M \equiv_{\Sigma_{n+2}} \sup_{N}(M)$. 
\end{enumerate}
\end{prop}

\section{Parameter-free collection schemes}\label{Sec:Coll-}

In this section, we prove the model theoretic characterization of the scheme $\CWm(\Sigma_{n+1})$. 
As in the previous sections, we introduce several notions. 

\begin{defn}
Let $M, K \models \PA^-$ be such that $M \subseteq K$ and $n \geq 0$. 
\begin{itemize}
    \item $M \prec_{\Sigma_{n+1}}^< K$ iff for any $\vec{a} \in M$ and any $\Pi_n$ formula $\varphi(\vec{x})$, we have $M \models \exists \vec{x} < \vec{a}\, \varphi(\vec{x})$ if and only if $K \models \exists \vec{x} < \vec{a}\, \varphi(\vec{x})$. 
\end{itemize}
\end{defn}

\begin{defn}
Let $M \models \PA^-$ and $n \geq 0$. 
\begin{itemize}
    \item We say that $M$ satisfies the condition $\eexf{n+1}$ iff for any $N \models \PA^-$, if $M \prec N$, then $\sup_N(M) \prec_{\Sigma_{n+1}}^< N$. 

    \item We say that $M$ satisfies the condition $\coff{n+1}$ iff for any $N \models \PA^-$, if $M \prec N$, then $M \prec_{\Sigma_{n+1}}^< \sup_N(M)$. 

    \item We say that $M$ satisfies the condition $\COFf{n+1}$ iff for any $K \models \PA^-$, if $M \dcf K$, then $M \prec_{\Sigma_{n+1}}^< K$. 

\end{itemize}
\end{defn}

It is easy to see that $\cof{n+1}$ (resp.~$\COF{n+1}$) implies $\coff{n+1}$ (resp.~$\COFf{n+1}$), and $\coff{n+1}$ (resp.~$\COFf{n+1}$) implies $\cofm{n+1}$ (resp.~$\COFm{n+1}$). 
By the following proposition, we may focus only on the conditions $\coff{n+1}$ and $\COFf{n+1}$. 

\begin{prop}
For any $M \models \PA^-$ and $n \geq 0$, if $M$ satisfies $\eex{n}$, then $M$ also satisfies $\eexf{n+1}$. 
\end{prop}
\begin{proof}
Suppose that $M$ satisfies $\eex{n}$. 
Let $N \models \PA^-$ be any model such that $M \prec N$. 
Let $\varphi(\vec{x})$ be any $\Pi_n$ formula and $\vec{a} \in \sup_N(M)$. 
Suppose $N \models \exists \vec{x} < \vec{a}\, \varphi(\vec{x})$, then we find some $\vec{b} \in N$ such that $N \models \vec{b} < \vec{a} \land \varphi(\vec{b})$. 
Since $\sup_N(M) \ee N$, we have $\vec{b} \in \sup_N(M)$. 
So, we obtain $\sup_N(M) \models \varphi(\vec{b})$ because $\sup_N(M) \prec_{\Sigma_n} N$ by $\eex{n}$. 
We conclude $\sup_N(M) \models \exists \vec{x} < \vec{a}\, \varphi(\vec{x})$. 
The converse direction directly follows from $\eex{n}$. 
\end{proof}

We prove the following characterization theorem. 

\begin{thm}\label{MT4}
For any $M \models \PA^-$ and $n \geq 0$, the following are equivalent: 
\begin{enumerate}
    \item $M \models \Cd(\Sigma_{n+1})$. 

    \item $M \models \Cm(\Sigma_{n+1})$. 

    \item $M$ satisfies the conditions $\eex{n}$ and $\COFf{n+2}$. 

    \item $M$ satisfies the conditions $\eex{n}$ and $\coff{n+2}$. 
\end{enumerate}
\end{thm}
\begin{proof}
$(1 \Rightarrow 2)$: Trivial. 

$(2 \Rightarrow 3)$: Suppose $M \models \Cm(\Sigma_{n+1})$. 
We have that $M$ satisfies $\eex{n}$ as in the proof of Theorem \ref{MT3}. 
Let $K \models \PA^-$ be such that $M \dcf K$. 
We show that $M \prec_{\Sigma_{n+2}}^< K$. 
By Corollary \ref{eexcof2}, we have $M \prec_{\Sigma_{n+1}} K$. 

Let $\vec{a} \in M$ and $\varphi(\vec{x}, \vec{y})$ be any $\Sigma_{n}$ formula such that $K \models \exists \vec{x} < \vec{a}\, \forall \vec{y}\, \varphi(\vec{x}, \vec{y})$. 
By the same argument as in the proof of Theorem \ref{MT3}, we have that $M$ satisfies $\forall \vec{v}\, \exists \vec{x} < \vec{a}\, \forall \vec{y} < \vec{v}\, \varphi(\vec{x}, \vec{y})$.
By applying $\Cm(\Sigma_{n+1})$, we conclude that $M \models \exists \vec{x} < \vec{a}\, \forall \vec{y}\, \varphi(\vec{x}, \vec{y})$. 

$(3 \Rightarrow 4)$: Trivial. 

$(4 \Rightarrow 1$): Suppose that $M$ satisfies the conditions $\eex{n}$ and $\coff{n+2}$. 
We prove that the contrapositive of each instance of $\Cd(\Sigma_{n+1})$ holds in $M$. 
For any $\vec{a} \in M$ and any $\Sigma_{n+1}$ formula $\varphi(\vec{x}, \vec{y}, \vec{z})$, we assume
\[
	M \models \exists \vec{z}\, \forall \vec{v}\, \exists \vec{x} < \vec{a}\, \forall \vec{y} < \vec{v}\, \neg \varphi(\vec{x}, \vec{y}, \vec{z}). 
\]
So, for some $\vec{b} \in M$, we have
\[
	M \models \forall \vec{v}\, \exists \vec{x} < \vec{a}\, \forall \vec{y} < \vec{v}\, \neg \varphi(\vec{x}, \vec{y}, \vec{b}). 
\]
By the same argument as in the proof of Theorem \ref{MT2}, we obtain that $\sup_N(M) \models \exists \vec{x} < \vec{a}\, \forall \vec{y}\, \neg \varphi(\vec{x}, \vec{y}, \vec{b})$ by using the condition $\eex{n}$. 
Then, for some $\vec{d} \in M$, we have $\sup_N(M) \models \exists \vec{z} < \vec{d}\, \exists \vec{x} < \vec{a}\, \forall \vec{y}\, \neg \varphi(\vec{x}, \vec{y}, \vec{z})$. 
Since $M \prec_{\Sigma_{n+2}}^< \sup_N(M)$ by the condition $\coff{n+2}$, we get $M \models \exists \vec{z} < \vec{d}\, \exists \vec{x} < \vec{a}\, \forall \vec{y}\, \neg \varphi(\vec{x}, \vec{y}, \vec{z})$. 
Thus, $M \models \exists \vec{z}\, \exists \vec{x} < \vec{a}\, \forall \vec{y}\, \neg \varphi(\vec{x}, \vec{y}, \vec{z})$. 
\end{proof}

We propose the following problems. 

\begin{prob}\label{preserveProb2}
Let $n \geq 0$ and $M, K \models \PA^-$ be such that $M \dcf K$. 
\begin{enumerate}
    \item Does $M \models \Cm(\Sigma_{n+1})$ imply $K \models \Cm(\Sigma_{n+1})$?
    \item If $M$ satisfies $\coff{n+1}$, then does $K$ satisfy $\coff{n+1}$?
\end{enumerate}
\end{prob}

\begin{prob}
For $n \geq 0$, does $\PA^- + \CWm(\Sigma_{n+1})$ prove $\Cm(\Sigma_{n+1})$?
\end{prob}

Cord\'on-Franco et al.~\cite[Proposition 5.6]{CFL} showed that $\mathbf{B}\Sigma_{n+1}^- \nvdash \Cm(\Sigma_{n+1})$ if and only if $\mathbf{I}\Delta_0 + \Cm(\Sigma_{n+1})$ is not axiomatized by any set of Boolean combinations of $\Sigma_{n+2}$ sentences. 
This equivalence also follows from the proof of Proposition \ref{cofm_prop}. 
Relating to this problem, we get the following proposition. 

\begin{prop}
For any $n \geq 0$, the following are equivalent: 
\begin{enumerate}
    \item $\mathbf{B}\Sigma_{n+1}^- \nvdash \Cm(\Sigma_{n+1})$. 
    \item There exist $M, N \models \mathbf{I}\Sigma_n$ such that $M \prec N$, $N \neq \sup_N(M)$, $M \equiv_{\Sigma_{n+2}} \sup_N(M)$ and $M \not \prec_{\Sigma_{n+2}}^< \sup_N(M)$. 
\end{enumerate}
\end{prop}
\begin{proof}
    $(1 \Rightarrow 2)$: Suppose $\mathbf{B}\Sigma_{n+1}^- \nvdash \Cm(\Sigma_{n+1})$. 
    We obtain a model $M \models \mathbf{I}\Delta_0 + \CWm(\Sigma_{n+1})$ with $M \not \models \Cm(\Sigma_{n+1})$. 
    Then, $M$ is a model of $\mathbf{I}\Sigma_n$. 
    By Theorems \ref{MT3} and \ref{MT4}, $M$ satisfies $\eex{n}$ and $\COFm{n+2}$ but does not satisfy $\COFf{n+2}$. 
    Then, we obtain a model $N \models \mathbf{I}\Sigma_n$ such that $M \prec N$ and $M \not \prec_{\Sigma_{n+2}}^< \sup_N(M)$. 
    It follows from the condition $\COFm{n+2}$ that $M \equiv_{\Sigma_{n+2}} \sup_N(M)$. 
    Since $M \prec N$ and $M \not \prec_{\Sigma_{n+2}}^< \sup_N(M)$, we have $N \neq \sup_N(M)$. 

    $(2 \Rightarrow 1)$: Let $M, N \models \mathbf{I}\Sigma_n$ satisfy the conditions stated in Clause 2. 
    By Corollary \ref{CorMT1} and Proposition \ref{cofm_prop}, $M$ satisfies $\eex{n}$ and $\cofm{n+2}$. 
    Hence, by Theorem \ref{MT3}, $M \models \CWm(\Sigma_{n+1})$. 
    Since $M$ does not satisfy $\coff{n+2}$, by Theorem \ref{MT4}, we obtain $M \not \models \Cm(\Sigma_{n+1})$. 
    Therefore, we conclude $\mathbf{B}\Sigma_{n+1}^- \nvdash \Cm(\Sigma_{n+1})$. 
\end{proof}

We close this section with the following analogue of Propositions \ref{cof_prop2} and \ref{cofm_prop}. 

\begin{prop}\label{coff_prop}
Let $n \geq 0$. 
If $M \models \mathbf{I}\Delta_0$ satisfies $\eex{n}$, then the following are equivalent: 
\begin{enumerate}
    \item $M$ satisfies $\coff{n+2}$. 
    \item For any $N \models \PA^-$, if $M \prec N$, then there exists an $N' \models \PA^-$ such that $N \prec N'$ and $M \prec_{\Sigma_{n+2}}^< \sup_{N'}(M)$. 
    \item There exists an $N \models \PA^-$ such that $M \prec N$, $N \neq \sup_{N}(M)$ and $M \prec_{\Sigma_{n+2}}^< \sup_{N}(M)$. 
\end{enumerate}
\end{prop}
\begin{proof}
Let $M \models \mathbf{I}\Delta_0$ satisfy $\eex{n}$. 
By Corollary \ref{CorMT1}, we have $M \models \mathbf{I}\Sigma_{n}$. 

$(1 \Rightarrow 2)$ and $(2 \Rightarrow 3)$ are proved in the similar way as in the proof of Proposition \ref{cof_prop2}. 

$(3 \Rightarrow 1)$: Let $N \models \PA^-$ be such that $M \prec N$, $N \neq \sup_N(M)$ and $M \prec_{\Sigma_{n+2}} \sup_N(M)$. 
By Theorem \ref{MT4}, it suffices to show $M \models \Cm(\Sigma_{n+1})$. 
Let $\vec{a} \in M$ and $\varphi(\vec{x}, \vec{y})$ be any $\Pi_n$ formula such that $M \models \forall \vec{x} < \vec{a}\, \exists \vec{y}\, \varphi(\vec{x}, \vec{y})$. 
Since $M \prec_{\Sigma_{n+2}}^< \sup_N(M)$, we have $\sup_N(M) \models \forall \vec{x} < \vec{a}\, \exists \vec{y}\, \varphi(\vec{x}, \vec{y})$. 
As in the proof of Proposition \ref{cof_prop2}, we obtain $\sup_{N}(M) \models \mathbf{B}\Sigma_{n+1}$, and hence $\sup_N(M) \models \exists \vec{v}\, \forall \vec{x} < \vec{a}\, \exists \vec{y} < \vec{v}\, \varphi(\vec{x}, \vec{y})$. 
If $n = 0$, this formula is $\Sigma_1$. 
If $n \geq 1$, it can also be regarded as $\Sigma_{n+1}$ because both $M$ and $\sup_N(M)$ are models of $\Coll(\Sigma_n)$. 
By Corollary \ref{eexcof2}, we have $M \prec_{\Sigma_{n+1}} \sup_N(M)$, and hence $M \models \exists \vec{v}\, \forall \vec{x} < \vec{a}\, \exists \vec{y} < \vec{v}\, \varphi(\vec{x}, \vec{y})$. 
We are done. 
\end{proof}

\section{$\cof{n+1}$ versus $\COF{n+1}$}\label{Sec:versus}

Corollary \ref{cofCOF} states that if $M \models \PA^-$ satisfies $\eex{n}$, then $\cof{n+2}$ and $\COF{n+2}$ are equivalent for $M$. 
So, $\cof{2}$ and $\COF{2}$ are equivalent. 

Then, we propose the following problem. 

\begin{prob}\label{cofCOFProb2}
    For $n \geq 0$ and $M \models \PA^-$, are $\cof{n+3}$ and $\COF{n+3}$ equivalent?
\end{prob}

In the case of $M \models \mathbf{I}\Delta_0 + \exp$, an improvement of Corollary \ref{cofCOF} follows from Belanger and Wong's theorem (Theorem \ref{BW}). 
In the proof of our improvement, we use the following lemma. 

\begin{lem}\label{ip}
For any $M, K \models \PA^-$, $M \prec_{\Sigma_1} K$ if and only if there exists an $N \models \PA^-$ such that $M \prec N$ and $K \prec_{\Delta_0} N$. 
\end{lem}
\begin{proof}
The implication $(\Leftarrow)$ is obvious. 
To show the implication $(\Rightarrow)$, is suffices to prove the consistency of the theory $\ElemDiag(M) \cup \Delta_0\text{-}\ElemDiag(K)$, which follows from $M \prec_{\Sigma_1} K$. 
\end{proof}

\begin{prop}\label{cofCOF2}
For any $n \geq 0$ and $M \models \mathbf{I}\Delta_0 + \exp$ satisfying $\eex{n}$, $M$ satisfies $\cof{n+3}$ if and only if $M$ satisfies $\COF{n+3}$. 
\end{prop}
\begin{proof}
Suppose $M \models \mathbf{I}\Delta_0 + \exp$ satisfies $\eex{n}$ and $\cof{n+3}$. 
By Theorem \ref{MT2}, $M \models \mathbf{B}\Sigma_{n+1} + \exp$.
Let $K \models \PA^-$ be such that $M \dcf K$. 
By Proposition \ref{GA0}, we have $M \prec_{\Sigma_1} K$, and so we get an $N \models \PA^-$ such that $M \prec N$ and $K \prec_{\Delta_0} N$ by Lemma \ref{ip}. 
Then, by Theorem \ref{BW}, we obtain $K \models \mathbf{B}\Sigma_{n+1}$, and hence $K$ satisfies $\COF{n+2}$ by Theorem \ref{MT2} again. 
Therefore, $K \prec_{\Sigma_{n+2}} \sup_N(K) = \sup_N(M)$. 
Also by $\cof{n+3}$ for $M$, we have $M \prec_{\Sigma_{n+3}} \sup_N(M)$. 
By combining them, we obtain $M \prec_{\Sigma_{n+3}} K$. 
We have shown that $M$ satisfies $\COF{n+3}$. 
\end{proof}

As a consequence, $\cof{3}$ and $\COF{3}$ are equivalent in the case of $M \models \mathbf{I}\Delta_0 + \exp$. 

Recently, the following interesting theorem is announced by Mengzhou Sun and Tin Lok Wong. 

\begin{thm}[Sun and Wong]\label{MW}
Let $n \geq 0$. 
\begin{enumerate}
    \item For any countable model $M \models \mathbf{B}\Sigma_{n+1} + \exp + \neg \mathbf{I}\Sigma_{n+1}$,  we have that $M$ does not satisfy $\cof{n+4}$. 
    \item There exists a countable model $M \models \mathbf{B}\Sigma_{n+1} + \exp + \neg \mathbf{I}\Sigma_{n+1}$ that satisfies $\COF{n+3}$. 
    \item There exists a uncountable model $M \models \mathbf{B}\Sigma_{n+1} + \exp + \neg \mathbf{I}\Sigma_{n+1}$ that satisfies $\COF{k}$ for all $k \geq 1$. 
\end{enumerate}
\end{thm}

The following proposition is obtained from the first clause of Theorem \ref{MW}. 

\begin{prop}\label{MW2}
    For $n \geq 0$ and any countable model $M \models \mathbf{I}\Delta_0 + \exp$, if $M$ satisfies $\cof{n+3}$, then $M \models \mathbf{B}\Sigma_{n+1}$. 
\end{prop}
\begin{proof}
    We prove the proposition by induction on $n$. 
    The case of $n = 0$ follows from Theorem \ref{MT2}.  
    We suppose that the statement holds for $n$, and let $M$ be any countable model of $M \models \mathbf{I}\Delta_0 + \exp$ satisfying $\cof{n+4}$. 
    By the induction hypothesis, $M \models \mathbf{B}\Sigma_{n+1} + \exp$.
    So, by Theorem \ref{MW}, we obtain that $M \models \mathbf{I}\Sigma_{n+1}$. 
    By Corollary \ref{CorMT1}, $M$ satisfies $\eex{n+1}$, and so by Theorem \ref{MT2}, $M \models \mathbf{B}\Sigma_{n+2}$. 
\end{proof}

Proposition \ref{MW2} gives us the following affirmative answer to Problem \ref{preserveProb} in the case that $M$ is a countable model of $\mathbf{I} \Delta_0 + \exp$.  

\begin{prop}\label{cofCOF3}
For $n \geq 0$ and countable $M \models \mathbf{I}\Delta_0 + \exp$, we have that $\cof{n+4}$ and $\COF{n + 4}$ are equivalent. 
\end{prop}
\begin{proof}
    Suppose that $M \models \mathbf{I}\Delta_0 + \exp$ is countable and satisfies $\cof{n+4}$. 
    By Proposition \ref{MW2}, $M \models \mathbf{I}\Sigma_{n+1}$. 
    By Corollary \ref{CorMT1}, $M$ satisfies $\eex{n+1}$. 
    Then, by Proposition \ref{cofCOF2}, $M$ satisfies $\COF{n+4}$. 
\end{proof}

The situation of the implications on properties for countable models of $\mathbf{I}\Delta_0 + \exp$ is visualized in Figure \ref{Fig3}. 
The second clause of Theorem \ref{MW} together with the facts that $\mathbf{B}\Sigma_{n+1} + \exp \nvdash \mathbf{I}\Sigma_{n+1}$ and $\mathbf{I}\Sigma_{n} + \exp \nvdash \mathbf{B} \Sigma_{n+1}$ shows that no more arrows can be added to the diagram. 
Also, the countability of models cannot be removed in Figure \ref{Fig3} because of the third clause of Theorem \ref{MW}. 

\begin{figure}[ht]
\centering
\begin{tikzpicture}
\node (B1) at (-3, 0) {$\mathbf{B}\Sigma_{1}$};
\node (c2) at (0, 0) {$\cof{2}$};
\node (C2) at (3, 0) {$\COF{2}$};
\node (I1) at (-3, 1) {$\mathbf{I}\Sigma_{1}$};
\node (e1) at (-1.5, 1) {$\eex{1}$};
\node (c3) at (1.5, 1) {$\cof{3}$};
\node (C3) at (3, 1) {$\COF{3}$};
\node (B2) at (-3, 2) {$\mathbf{B}\Sigma_{2}$};
\node (e1c3) at (0, 2) {$\eex{1}$\ \&\ $\cof{3}$};
\node (I2) at (-3, 3) {$\mathbf{I}\Sigma_{2}$};
\node (e2) at (-1.5, 3) {$\eex{2}$};
\node (c4) at (1.5, 3) {$\cof{4}$};
\node (C4) at (3, 3) {$\COF{4}$};
\node (B3) at (-3, 4) {$\mathbf{B}\Sigma_{3}$};
\node (e2c4) at (0, 4) {$\eex{2}$\ \&\ $\cof{4}$};

\draw [<->, double] (B1)--(c2);
\draw [<->, double] (c2)--(C2);

\draw [->, double] (e1)--(c2);
\draw [->, double] (c3)--(c2);

\draw [<->, double] (I1)--(e1);
\draw [<->, double] (c3)--(C3);

\draw [->, double] (e1c3)--(e1);
\draw [->, double] (e1c3)--(c3);

\draw [<->, double] (B2)--(e1c3);

\draw [->, double] (e2)--(e1c3);
\draw [->, double] (c4)--(e1c3);

\draw [<->, double] (I2)--(e2);
\draw [<->, double] (c4)--(C4);

\draw [->, double] (e2c4)--(e2);
\draw [->, double] (e2c4)--(c4);

\draw [<->, double] (B3)--(e2c4);

\draw [->, double] (-1.1, 4.7)--(e2c4);
\draw [->, double] (1.1, 4.7)--(e2c4);

\end{tikzpicture}
\caption{Implications for countable models of $\mathbf{I} \Delta_0 + \exp$}\label{Fig3}
\end{figure}

\section*{Acknowledgement}

We would like to thank Tin Lok Wong for giving us a lot of valuable information on recent developments.
We would also like to thank Mengzhou Sun for his helpful comment. 
The first author was supported by JSPS KAKENHI Grant Numbers JP19K14586 and JP23K03200.

\bibliographystyle{plain}
\bibliography{ref}

\end{document}